\documentclass[12pt]{amsart}
\usepackage{amssymb,amsmath,amsthm,latexsym}
\usepackage{mathrsfs}
\usepackage{a4wide}
\usepackage{eucal}
\usepackage[noadjust]{cite}
\usepackage[usenames,dvipsnames]{color}
\usepackage{stackengine}
\usepackage[utf8]{inputenc}
\usepackage[T1]{fontenc}
\usepackage{dsfont}
\DeclareFontFamily{U}{mathx}{\hyphenchar\font45}
\DeclareFontShape{U}{mathx}{m}{n}{
      <5> <6> <7> <8> <9> <10>
      <10.95> <12> <14.4> <17.28> <20.74> <24.88>
      mathx10
      }{}
\DeclareSymbolFont{mathx}{U}{mathx}{m}{n}
\DeclareFontSubstitution{U}{mathx}{m}{n}
\DeclareMathAccent{\widecheck}{0}{mathx}{"71}
\DeclareMathAccent{\wideparen}{0}{mathx}{"75}

\newtheorem{theorem}{Theorem}[section]
\newtheorem{lemma}[theorem]{Lemma}
\newtheorem{corollary}[theorem]{Corollary}
\newtheorem{proposition}[theorem]{Proposition}
\theoremstyle{remark}

\newtheorem{remark}[theorem]{Remark}
\newtheorem*{remark*}{Remark}
\theoremstyle{definition}
\newtheorem{definition}[theorem]{Definition}

\numberwithin{equation}{section}
\makeatother

\newcommand{\vertiii}[1]{{\left\vert\kern-0.25ex\left\vert\kern-0.25ex\left\vert #1 
    \right\vert\kern-0.25ex\right\vert\kern-0.25ex\right\vert}}

\newcounter{smallromans}

\newenvironment{romanenumerate}
{\begin{list}{{\normalfont\textrm{(\roman{smallromans})}}}%
  {\usecounter{smallromans}\setlength{\itemindent}{0cm}%
   \setlength{\leftmargin}{5.5ex}\setlength{\labelwidth}{5.5ex}%
   \setlength{\topsep}{.5ex}\setlength{\partopsep}{.5ex}%
   \setlength{\itemsep}{0.1ex}}}%
{\end{list}}

\newcounter{smallromansdash}

{\end{list}}

\newcounter{bigromans} 
  {\end{list}}

\begin{document}
\date{October 7, 2017}
\title{When is multiplication in a Banach algebra open?}
\author[Sz.~Draga]{Szymon Draga}
\address{Institute of Mathematics,
University of Silesia,
Bankowa 14, 40-007 Katowice, Poland and Institute of Mathematics,
Czech Academy of Sciences, \v{Z}itn\'{a} 25, 115 67 Prague~1, Czech Republic}
\email{szymon.draga@gmail.com}

\author[T.~Kania]{Tomasz Kania}
\address{Mathematics Institute,
University of Warwick,
Gibbet Hill Rd, 
Coventry, CV4 7AL, 
England and Institute of Mathematics, Czech Academy of Sciences, \v{Z}itn\'{a} 25, 115~67 Prague 1, Czech Republic}
\email{tomasz.marcin.kania@gmail.com}


\thanks{Research of the first-named author was supported by the GA\v{C}R project 16-34860L and RVO: 67985840. The second-named author acknowledges with thanks funding received from the European Research Council; ERC Grant Agreement No.~291497.}
\begin{abstract}We develop the theory of Banach algebras whose multiplication (regarded as a bilinear map) is open. We demonstrate that such algebras must have topological stable rank 1, however the latter condition is strictly weaker and implies only that products of non-empty open sets have non-empty interior. We then investigate openness of convolution in semigroup algebras resolving in the negative a problem of whether convolution in $\ell_1(\mathbb{N}_0)$ is open. By appealing to ultraproduct techniques, we demonstrate that neither in $\ell_1(\mathbb{Z})$ nor in $\ell_1(\mathbb Q)$ convolution is uniformly open. The problem of openness of multiplication in Banach algebras of bounded operators on Banach spaces and their Calkin algebras is also discussed.\end{abstract}
\maketitle

\section{Introduction}

Perhaps the most spectacular failure of the multilinear version of Banach's open-mapping theorem is elucidated by D.H.~Fremlin's example of the function $f(x)=x-1/2$ ($x\in [0,1]$) that witnesses the lack of openness of multiplication in the real Banach algebra $C[0,1]$ at $(f,f)$ for rather obvious reasons. Of course, examples of bounded bilinear maps that are not open had been known before (just to mention the primordial example of Cohen \cite{cohen} and the first example on a finite-dimensional space by Horowitz \cite{horowitz}), however it was Fremlin's example that triggered intensive study of openness of multiplication in various function algebras \cite{bbs, bm, bms, bmw, bww, be1, be2, be3, be4, kt, komisarski}.\smallskip

The aim of this paper is twofold. Firstly, we aim at putting the theory of Banach algebras with open multiplication at a more systematic footing. By this we mean investigation of preservation of (weakly, uniformly) open multiplication by operations such as completions, direct products, quotients, direct limits, ultraproducts, \emph{etc}. We wish to emphasise that our results place the property of having topological stable rank 1 (that is, having dense invertible group) strictly between having weakly open and open multiplication. \smallskip

Secondly, once the above-listed permanence properties are established, we turn our attention to concrete examples of Banach algebras such as algebras of continuous functions, semigroup convolution algebras, algebras of functions having bounded variation or algebras of bounded operators on Banach spaces. For example, we extend Komisarski's result that links openness of multiplication in $C(X)$ with the covering dimension of $X$ to the case of complex scalars (Proposition~\ref{dim2}); as an offshoot of these investigations, we obtain a new proof in the case where $X$ is zero-dimensional (Proposition~\ref{0dim}). Moreover, we address various questions left open in the aforementioned papers. For instance, we prove that the Cauchy product (convolution) in $\ell_1(\mathbb{N}_0)$ is not open (Corollary~\ref{Cauchyproduct}), thereby answering in the negative a question posed by Balcerzak, Behrends and Strobin \cite[Question 3 on p.~493]{bbs}. Moreover, using ultraproduct methods and the theory of infinite Abelian groups, we prove that the algebras $\ell_1(\mathbb{Z})$ and $\ell_1(\mathbb{Q})$ do not have uniformly open multiplication (Corollary~\ref{ZandQ}). We also recover a recent result of Kowalczyk and Turowska \cite{kt} asserting that multiplication in the Banach algebra of all scalar-valued functions with bounded variation on an interval is weakly open.\smallskip

The final section is devoted to Banach algebras $\mathcal{B}(E)$ of bounded operators acting on a~Banach space $E$. We demonstrate that for all classical Banach spaces $E$, neither $\mathcal{B}(E)$ nor the corresponding Calkin algebra has open multiplication (Theorem~\ref{B(E)andQ(E)}). However, we also show, by using Fredholm theory, that these algebras have weakly open multiplication for certain Banach spaces $E$, that include the hereditarily indecomposable Banach spaces (Theorem~\ref{B(E)weaklyopen}).\smallskip

Let us remark that the problem of openness of a bilinear map has global rather than local character. Indeed, it is clear that if multiplication is open when restricted to a~ball centred at zero (more precisely, to the set $B(0,r)\times B(0,r)$ for some $r>0$), then it is open. However, openness on a~different open ball is usually not sufficient. Indeed, multiplication in a~Banach algebra with the identity element $1$ is always open on the ball $B(1,1)$ comprising invertible elements; this follows from the fact that the mapping $x\mapsto xy$ is a~homeomorphism for any invertible element $y$.

\section{Preliminaries}
\subsection{Open and uniformly open maps} The notion of an open map between metric spaces plays a pivotal r\^{o}le in this article so let us invoke the definition. \smallskip

Let $X$ and $Y$ be metric spaces. We denote by $B(x,r)$ the open ball in $X$ centred at $x$ with radius $r$ and we use the same symbol for open balls in $Y$ trusting that it will not lead to confusion. (Occasionally, we may use the subscripts $X, Y$ \emph{etc.}, to indicate the space we consider the ball in.) A function $f\colon X\to Y$ is \emph{open} whenever for any open set $U\subseteq X$ the image $f(U)$ is open. This is of course equivalent to saying that for any $x\in X$ and $\varepsilon > 0$ there is $\delta > 0$ such that
$$B(f(x), \delta)\subseteq f[B(x,\varepsilon)].$$
When $\delta$ depends on $\varepsilon$ but not on $x$ such map is termed \emph{uniformly open}. Of course, surjective, bounded linear maps between Banach spaces are important examples of open maps. We shall require an open-mapping theorem due to Schauder, which says essentially that uniformly \emph{almost} open maps on complete metric spaces are (uniformly) open. (It may be found, \emph{e.g.}, in \cite[Lemma 3.9]{mv}.)\smallskip

\begin{lemma}[Schauder]\label{dek}Let $X$ and $Y$ be metric spaces and let $f\colon X\to Y$ be a continuous function. If $X$ is complete and $f$ has the property that for every $\varepsilon > 0$ there exists $\delta > 0$ such that for any $x\in X$ one has $$B(f(x), \delta)\subseteq \overline{f[B(x, \varepsilon)]},$$ then $f$ is uniformly open.\end{lemma}

Another related notion concerning maps between metric spaces is weak openness. A~map $T\colon X\to Y$ is \emph{weakly open} whenever for any non-empty open set $U\subseteq X$ the image $T[U]$ has non-empty interior (\emph{cf}.~\cite[Section 2]{bww}). We also say that maps $T_\gamma\colon X_\gamma\to Y_\gamma$ ($\gamma\in \Gamma$) are \emph{equi-uniformly open} if they are uniformly open and $\delta$ depends on $\varepsilon>0$ in the same way for all $\gamma$.

\subsection{Banach algebras} A \emph{Banach algebra} is a real or complex Banach space $A$ furnished with associative multiplication making it a ring that satisfies $\|ab\|\leqslant \|a\|\!\cdot\!\|b\|$ ($a,b\in A$). When we do not specify whether we talk about a~real or complex Banach algebra, it means that the statement is valid for both choices of the scalar field---otherwise, we clearly indicate the underlying scalar field. A Banach algebra is \emph{unital} if it has the multiplicative identity. Banach algebras are \emph{isomorphic} if there exists an invertible, bounded linear map between them that preserves multiplication. Let $A$ be a unital Banach algebra. We denote by ${\rm GL}\,A$ the set of all invertible elements in $A$; this set is a group under multiplication and it is open in the norm topology. An element $a$ in a Banach algebra $A$ is a \emph{topological zero divisor} if $$\inf\{ \|xa\|+\|ax\|\colon x\in A,\|x\|=1\}=0.$$

We shall require the following elementary lemma concerning topological zero divisors in Banach algebras. It may be found, \emph{e.g.}, in \cite[Theorem 14 on p.~13]{bd}.
\begin{lemma}\label{topzero}Let $A$ be a unital Banach algebra. Then the boundary of ${\rm GL}\,A$ consists of topological zero divisors. \end{lemma}

Let $A$ be a unital Banach algebra. An element $a\in A$ is \emph{left-invertible} (or \emph{right-invertible}) if there exists $x\in A$ such that $xa=1$ (or $ax=1$). Conspicuously, these notions coincide in commutative algebras. However, in general they are different (for example, the right-shift operator on the Hilbert space $\ell_2(\mathbb{N})$ is left- but not right-invertible). It turns out that density of invertible, left-invertible and right-invertible elements in a unital Banach algebra are all the same. This has been noticed for C*-algebras by Robertson \cite{robertson} and proved in full generality by Rieffel \cite[Proposition 3.1]{rieffel}. Before we state the result, let us introduce the notion of topological stable rank 1.

\begin{definition}A unital Banach algebra has \emph{topological stable rank} 1 $({\rm tsr}\,A = 1$, for short), when ${\rm GL}\,A$ is dense in $A$. \end{definition}
\begin{proposition}\label{rieffel}Let $A$ be a unital Banach algebra. Then the following conditions are equivalent:
\begin{romanenumerate}
\item ${\rm tsr}\,A=1$;
\item the set of all left-invertible elements of $A$ is dense in $A$;
\item the set of all right-invertible elements of $A$ is dense in $A$.
\end{romanenumerate}
Moreover, if any of the above-listed conditions is met, then $a\in A$ is left-invertible if and only if it is right-invertible (hence invertible).
\end{proposition}

We shall require the following standard fact in Banach algebras, which follows from the observation that elements whose resolvent sets have empty interior in the complex plane form a dense set.
\begin{proposition}\label{index0}Suppose that $A$ is a unital Banach algebra with the property that each element $a\in A$ has totally disconnected spectrum. Then ${\rm tsr}\, A=1$. \end{proposition}
In particular, the above applies to finite-dimensional unital Banach algebras. We remark in passing that a more general result for higher stable ranks may be found in \cite{cg}. The following fact is an immediate corollary to Proposition~\ref{index0}.

\begin{corollary}Suppose that $A$ is a unital commutative Banach algebra such that the linear span of idempotents in $A$ is dense. Then ${\rm tsr}\,A=1$. \end{corollary}

We refer to \cite{df} and \cite{misa} for further results concerning Banach algebras having topological stable rank 1.

\subsection{Semigroup algebras} Let $S$ be a non-empty set and let $\ell_1(S)$ denote the space of all absolutely summable scalar-valued functions on $S$. By $(e_s)_{s\in S}$ we denote the canonical unit vector basis of $\ell_1(S)$. Suppose that $S$ is a semigroup, that is $S$ is a set that is furnished with an associative operation $\cdot\,$. Then $\ell_1(S)$ becomes a Banach algebra with the convolution product $\ast$ (which we will sometimes refer to as multiplication in $\ell_1(S)$):
$$x\ast y =\sum_{t\in S} \Big(\sum_{r\cdot s=t}\xi_r \eta_s\Big)e_t, $$
where $x=(\xi_s)_{s\in S}, y=(\eta_s)_{s\in S}\in \ell_1(S)$. Denoting by $\mathbb{N}_0$ the additive semigroup of all non-negative integers, convolution in $\ell_1(\mathbb{N}_0)$ is the familiar Cauchy product of series. The semigroup of $\mathbb{N}_0$ is a paradigm example of a \emph{cancellative semigroup}, \emph{i.e.}, a semigroup $S$ with the property that for all $g\in S$ the maps $s\mapsto gs, s\mapsto sg$ ($s\in S$) are injective.\smallskip

If $T\subseteq S$, then $\ell_1(T)$ is naturally a closed subspace of $\ell_1(S)$, which becomes a subalgebra as long as $T$ is a subsemigroup of $S$. If $S$ is a \emph{monoid}, that is, if $S$ has a neutral element $1\in S$, then the Banach algebra $\ell_1(S)$ is unital, where $e_1$ is the identity element; the converse is not true.\smallskip

Every surjective homomorphism $\vartheta\colon T\to S$ of semigroups induces a surjective homomorphism $h_\vartheta\colon \ell_1(T)\to \ell_1(S)$ of Banach algebras by the action \begin{equation}\label{homomorphismtheta}h_\vartheta e_t = e_{\vartheta(t)}\quad (t\in T).\end{equation} We refer the reader to \cite[Chapter 4]{dls} for more information on semigroup algebras.

\subsubsection{The space $c_0(\Lambda)$} Let $\Lambda$ be a non-empty set. We denote by $c_0(\Lambda)$ the space of all scalar-valued functions on $\Lambda$ that converge to 0, that is, functions $f$ such that for any $\varepsilon>0$ the set $\{\lambda\in \Lambda\colon |f(\lambda)|\geqslant \varepsilon
\}$ is finite. The space $c_0(\Lambda)$ is a Banach space when equipped with the supremum norm and $c_0(\Lambda)^*$ is isometrically isomorphic to $\ell_1(\Gamma)$ via the pairing
$$\langle f,g\rangle = \sum_{\lambda\in \Lambda}f(\lambda)g(\lambda)\quad \big(f\in c_0(\Lambda), g\in \ell_1(\Lambda)\big).$$

\subsubsection{Ultraproducts of group algebras} In this section we provide a link between the algebraic ultraproduct of groups and the corresponding Banach-space ultraproduct of their group algebras. This has been essentially developed by Daws in \cite[Section 5.4]{daws} in the case where all groups are the same but the constructions carry over to arbitrary ultraproducts so we only sketch them.\smallskip

We briefly define the ultraproducts of Banach algebras. Let $(A_\gamma)_{\gamma\in \Gamma}$ be a family of Banach spaces. Then $A = ( \bigoplus_{\gamma\in \Gamma} A_\gamma)_{\ell_\infty(\Gamma)}$, the space of all tuples $(x_\gamma)_{\gamma\in \Gamma}$ with $x_\gamma\in A_\gamma$ ($\gamma\in \Gamma$) and $\sup_{\gamma\in \Gamma}\|x_\gamma\|<\infty$ is a Banach space under the supremum norm. Let us assume that the index set $\Gamma$ is infinite. We fix a non-principal ultrafilter $\mathcal{U}$ on $\Gamma$. Then the subspace $J=c_0^{\mathcal U}(A_\gamma)_{\gamma\in \Gamma}$ of $A$ comprising all tuples $(x_\gamma)_{\gamma\in \Gamma}$ such that $\lim_{\gamma\to \mathcal{U}}\|x_\gamma\|=0$ is closed. The (Banach-space) \emph{ultraproduct} $\prod_{\gamma\in \Gamma}^{\mathcal U}A_\gamma$ of $(A_\gamma)_{\gamma\in \Gamma}$ with respect to $\mathcal{U}$ is the quotient space $A / J$. If $A_\gamma$ ($\gamma\in \Gamma$) are Banach algebras, then naturally so is $A$ (endowed with the coordinate-wise product) and $J$ is then a closed ideal, hence an ultraproduct of Banach algebras is a Banach algebra.\smallskip

Let $(G_\gamma)_{\gamma\in \Gamma}$ be an infinite collection of groups and let $\mathcal U$ be an ultrafilter on the index set $\Gamma$. The algebraic \emph{ultraproduct} $\prod_{\gamma\in \Gamma}^{\mathcal{U}}G_\gamma$ (denoted also $G^{\mathcal U}$ when $G_\gamma=G$ are all equal and termed then the \emph{ultrapower} of $G$) is the direct product $\prod_{\gamma\in \Gamma}G_\gamma$ quotiented by the equivalence relation $$(g_\gamma)_{\gamma\in \Gamma} \sim (h_\gamma)_{\gamma\in \Gamma}\quad\text{ if and only if }\quad\{\gamma\in \Gamma\colon g_\gamma = h_\gamma\}\in \mathcal{U}.$$

The ultraproduct of groups carries an intrinsic structure of a group, which is Abelian if so are all $G_\gamma$. We define the maps
$$\varphi_0\colon c_0\Big({\prod_{\gamma\in \Gamma}}^{\mathcal{U}}G_\gamma \Big) \to {\prod_{\gamma\in \Gamma}}^{\mathcal{U}}c_0(G_\gamma),\qquad \varphi_1\colon \ell_1\Big({\prod_{\gamma\in \Gamma}}^{\mathcal{U}}G_\gamma \Big) \to {\prod_{\gamma\in \Gamma}}^{\mathcal{U}}\ell_1(G_\gamma) $$
by
$$\varphi_i\big(e_{[(g_\gamma)_{\gamma\in \Gamma}]}\big) = [(e_{g_\gamma})_{\gamma\in \Gamma}]\qquad\mbox{ \big($i=0,1$ and $({g_\gamma})_{\gamma\in \Gamma}\in \prod_{\gamma\in \Gamma}G_\gamma$}\big)$$
and extend them by linearity to (in general, non-surjective) isometries. Given the duality between the Banach spaces $c_0\Big(\prod_{\gamma\in \Gamma}^{\mathcal{U}}G_\gamma \Big)$ and $\ell_1\Big(\prod_{\gamma\in \Gamma}^{\mathcal{U}}G_\gamma \Big)$, the ultraproduct $\prod_{\gamma\in \Gamma}^{\mathcal{U}}\ell_1(G_\gamma)$ sits naturally as a subspace of the dual space of $\prod_{\gamma\in \Gamma}^{\mathcal{U}}c_0(G_\gamma)$. Consequently, we may view the composite map $\varphi_0^*\circ \varphi_1$ as the identity map on the space $\ell_1(\prod_{\gamma\in \Gamma}^{\mathcal{U}}G_\gamma )$ and so we may naturally identify this space with a~complemented subspace of $\prod_{\gamma\in \Gamma}^{\mathcal{U}}\ell_1(G_\gamma)$. It follows that the unique bounded linear map
\begin{equation}\label{homomorphismG}h_G\colon {\prod_{\gamma\in \Gamma}}^{\mathcal{U}}\ell_1(G_\gamma)\to \ell_1\Big({\prod_{\gamma\in \Gamma}}^{\mathcal{U}}G_\gamma \Big) \end{equation}
which satisfies $$h_G\Big(\big[(e_{g_\gamma})_{\gamma\in\Gamma}\big]\Big) = e_{[(g_\gamma)_{\gamma\in\Gamma}]}\qquad \Big(\big[(e_{g_\gamma})_{\gamma\in\Gamma}\big]\in {\prod_{\gamma\in \Gamma}}^{\mathcal U}\ell_1(G_\gamma)\Big)$$ is a contractive, surjective algebra homomorphism.

\subsection{Fredholm theory of operators on Banach spaces}
Let $E$ be a Banach space and denote by $\mathcal{B}(E)$ the Banach algebra of all bounded operators on $E$. The set $\mathcal{K}(E)$ of compact operators on $E$ is a closed ideal of $\mathcal{B}(E)$ and so we may form the quotient algebra $\mathcal{Q}(E)=\mathcal{B}(E)/\mathcal{K}(E)$, that we will often refer to as the \emph{Calkin algebra} of $E$. In this section we gather basic facts from Fredholm theory of operators on Banach spaces (we refer the reader to Arveson's book \cite[Chapter 3]{arveson} for a~beautiful exposition of the topic). \smallskip

An operator $T\in \mathcal{B}(E)$ is \emph{Fredholm} if the dimensions of $\ker T$ and $E / T[E]$ are finite, in which case we define the number $${\rm ind}\, T =  \dim \ker T - \dim E / T[E] \in \mathbb{Z},$$ called \emph{the Fredholm index} of $T$. The set comprising all Fredholm operators in $\mathcal{B}(E)$ is open. Certainly, every invertible operator has Fredholm index 0. Let us record the following simple lemma for the future reference.

\begin{lemma}\label{index0approx}Every index-zero Fredholm operator may be approximated by invertible operators. \end{lemma}
\begin{proof}Let $E$ be a Banach space and let $T\in \mathcal{B}(E)$ be a Fredholm operator of index 0. Let $P$ be a projection onto $\ker T$ and let $Q$ be a projection onto $E/T[E]\subseteq E$. Since $P[E]$ and $Q[E]$ have equal (finite) dimensions, we may find an isomorphism $U$ from $P[E]$ onto $Q[E]$. Then the operators $T + r UP$ ($r >0$) are invertible and converge to $T$ as $r\to 0^+$.\end{proof}

\section{Permanence properties of algebras with open multiplication}\label{permanence}
It is rather disappointing that openness of multiplication does not pass to closed subalgebras, even in the commutative case. An example of such phenomenon is exemplified by the algebra $C(\Delta)$ of all continuous functions on the Cantor set $\Delta$, which has uniformly open multiplication (Proposition~\ref{0dim}). Every compact metric space is a continuous image of $\Delta$, so there exists a continuous surjection $\theta\colon \Delta \to [0,1]^2$. The map $h_\theta (f)= f\circ \theta$ ($f\in C([0,1]^2)$) is an isometric algebra homomorphism, so the range of $h_\theta$ is isometrically isomorphic to $C([0,1]^2)$ but multiplication in this algebra is not open (see Remark~\ref{dim2}). Nevertheless, Banach algebras with (uniformly) open multiplication enjoy other permanence properties.

\subsection{Quotients, direct products and ultraproducts}
\begin{lemma}\label{quotient}Let $A$ and $B$ be Banach algebras and let $J\subseteq A$ be a closed ideal.
\begin{romanenumerate}
\item\label{it1} If $A$ has (uniformly, weakly) open multiplication, then so has $A/J$. Moreover, in both cases the dependence of $\delta$ from $\varepsilon$ in the quotient algebra is the same as in $A$.
\item\label{it2} If $h\colon A\to B$ is a surjective homomorphism and $A$ has (uniformly open) multiplication, then so has $B$.
\end{romanenumerate}
\end{lemma}

\begin{proof}
For \eqref{it1}, denote by $\kappa\colon A\to A/J$ be the canonical quotient map. By the very definition of the quotient norm we have $\kappa[B_A(x,r)]=B_{A/J}(\kappa(x),r)$
for any $r>0$. Now, if $\varepsilon>0$ and $\delta>0$ are such that $B_A(xy,\delta)\subseteq B_A(x,\varepsilon)\cdot B_A(y,\varepsilon)$
for some $x,y\in A$, then
$$
B_{A/J}(\kappa(xy),\delta)\subseteq B_{A/J}(\kappa(x),\varepsilon)\cdot B_{A/J}(\kappa(y),\varepsilon)
$$
which gives the assertion. \smallskip

As for weak openness, suppose that $B_A(x,\varepsilon) \cdot B_A(y,\varepsilon)$ has non-empty interior. As $\kappa$, being a surjection, is an open  map, $ B_{A/J}(\kappa(x),\varepsilon)\cdot B_{A/J}(\kappa(y),\varepsilon)$ has non-empty interior as well.\smallskip

For \eqref{it2}, since $B$ is isomorphic to $A/\ker h$, it is enough to prove the assertion in the case where $h$ is invertible. For $\varepsilon> 0$ and $\delta > 0$ corresponding to $\varepsilon / \|h\|$ in $A$ it is then enough to take $\delta^\prime = \delta / \|h^{-1}\|$.\end{proof}

\begin{lemma}\label{sums}Let $(A_\gamma)_{\gamma\in \Gamma}$ be a family of Banach algebras. Then multiplications in $A_\gamma$ $(\gamma\in \Gamma)$ are equi-uniformly open if and only if $A=\big(\! \bigoplus_{\gamma\in \Gamma} A_\gamma\big)_{\ell_\infty(\Gamma)}$ has uniformly open multiplication. Moreover, in the latter case multiplication in $A$ is uniformly open with $\delta$ depending on $\varepsilon >0$ in the same way as in $A_\gamma$ $(\gamma\in \Gamma)$.
\end{lemma}

\begin{proof}
First, assume that multiplications in $A_\gamma$ ($\gamma\in \Gamma$) are equi-uniformly open. Given $\varepsilon>0$, let $\delta>0$ be such that $B(x_\gamma y_\gamma,\delta)\subseteq B(x_\gamma,\varepsilon)\cdot B(y_\gamma,\varepsilon)$ for all $x_\gamma,y_\gamma\in A_\gamma$ ($\gamma\in \Gamma$). Consider $(x_\gamma)_{\gamma\in \Gamma},(y_\gamma)_{\gamma\in \Gamma}\in A$ and observe that
\begin{align*}
B\Big((x_\gamma y_\gamma)_{\gamma\in \Gamma},\frac{\delta}{\eta^2}\Big) &\subseteq\prod_{\gamma\in \Gamma} B\Big(x_\gamma y_\gamma,\frac{\delta}{\eta^2}\Big)\subseteq\prod_{\gamma\in \Gamma} B\Big(x_\gamma,\frac{\varepsilon}{\eta}\Big)\cdot B\Big(y_\gamma,\frac{\varepsilon}{\eta}\Big)\\
&\subseteq B\Big((x_\gamma)_{\gamma\in \Gamma},\varepsilon\Big)\cdot B\Big((y_\gamma)_{\gamma\in \Gamma},\varepsilon\Big)
\end{align*}
for any $\eta>1$. Hence
$$
B\Big((x_\gamma y_\gamma)_{\gamma\in \Gamma},\delta\Big)= \bigcup_{\eta>1}B\Big((x_\gamma y_\gamma)_{\gamma\in \Gamma},\frac{\delta}{\eta^2}\Big)\subseteq B\Big((x_\gamma)_{\gamma\in \Gamma},\varepsilon\Big)\cdot B\Big((y_\gamma)_{\gamma\in \Gamma},\varepsilon\Big).
$$

Now, assume that multiplication in $A$ is uniformly open. It is easy to see that the projection onto any coordinate is a surjective homomorphism of Banach algebras and Lemma \ref{quotient} gives the assertion.
\end{proof}

\begin{corollary}\label{ultraproductlemma}
Let $(A_\gamma)_{\gamma\in \Gamma}$ be a family of Banach algebras and let $\mathcal U$ be an ultrafilter on the set $\Gamma$. If multiplications in $A_\gamma$ $(\gamma\in \Gamma)$ are equi-uniformly open, then the ultraproduct $\prod_{\gamma\in \Gamma}^{\mathcal U} A_\gamma$ has uniformly open multiplication.
\end{corollary}
\begin{proof}The ultraproduct is, by the very definition, a quotient of $\big(\! \bigoplus_{\gamma\in \Gamma} A_\gamma\big)_{\ell_\infty(\Gamma)}$. \end{proof}

\subsection{Completions and direct limits}
Of course, one may talk about (uniform) openness of multiplication in normed algebras that are not necessarily complete. The following result says that completions of normed algebras with uniformly open multiplication have uniformly open multiplication too.

\begin{proposition}\label{completion}Suppose that $A$ is a Banach algebra that contains a dense subalgebra $A_0$ such that multiplication restricted to this subalgebra is uniformly open. Then $A$ has uniformly open multiplication. \end{proposition}

\begin{proof}Let $\varepsilon>0$ and choose $\delta >0$ corresponding  to $\varepsilon/2$. Take $x,y\in A$ and $x^\prime, y^\prime \in A_0$ such that $\|x-x^\prime\|, \|y-y^\prime\| < \varepsilon/2$ and $\|xy-x^\prime y^\prime\|<\delta / 2$. We have
$$ B(x^\prime y^\prime, \delta) \cap A_0 = B_{A_0}(x^\prime y^\prime, \delta) \subseteq  B_{A_0}(x^\prime, \varepsilon/2)\cdot B_{A_0}(y^\prime, \varepsilon/2) \\ 
 \subseteq   B(x, \varepsilon)\cdot B(y, \varepsilon)$$
so $$B(xy,\delta/2)\subseteq \overline{B(x^\prime y^\prime, \delta)} \subseteq \overline{B(x, \varepsilon)\cdot B(y, \varepsilon)}.$$ By Schauder's lemma (Lemma~\ref{dek}), it is uniformly open.\end{proof}

\begin{remark}In general, it is not possible to replace the hypothesis of uniform openness on a dense set by mere openness on a dense set to conclude openness everywhere. Indeed, let $A=M_n$ be the algebra of $n\times n$ matrices. Certainly, ${\rm GL}\, A$ is dense in $A$ as every matrix may be approached by a sequence of matrices that have non-zero eigen-values (see also Proposition~\ref{index0}). However, as ${\rm  GL}\,A$ is also open, multiplication restricted to the (dense) invertible group is open, yet it not open in the case where $n\geqslant 4$ as shown by Behrends \cite[p.~172 and Proposition 3.2(iv)]{be3}.\smallskip\end{remark}

\begin{corollary}Suppose that $(A_\gamma)_{\gamma\in \Gamma}$ is a net of Banach algebras directed by inclusion that have equi-uniformly open multiplications. Then the direct limit of $(A_\gamma)_{\gamma\in \Gamma}$ has uniformly open multiplication. \end{corollary}
\begin{proof}The normed algebra $\bigcup_{\gamma\in \Gamma} A_\gamma$ has uniformly open multiplication. \end{proof}

\section{The results}
The following result is a handy criterion for non-opennesses of multiplication in various Banach algebras.  

\begin{proposition}\label{multiplication}
Let $A$ be a unital Banach algebra. Suppose that multiplication in $A$ is an open mapping. Then ${\rm tsr}\,A = 1$.\end{proposition}

\begin{proof}
We proceed by contraposition so let us assume that ${\rm GL}\,A$ is not dense in $A$. Then, by Proposition~\ref{rieffel}, the set ${\rm RI}\,A$ of all right-invertible elements of $A$ is not dense in $A$. Take $f\in A\setminus\overline{{\rm RI}\,A}$. We \emph{claim} that multiplication in $A$ is not open at $(f,0)$.\smallskip

Indeed, by the hypothesis, there is $\varepsilon>0$ such that the intersection $B(f,\varepsilon)\cap\overline{{\rm RI}\,A}$ is empty. The set $B=B(f,\varepsilon)\cdot B(0,\varepsilon)$ consists of elements that are not right-invertible and, clearly, $0\in B$. However, $0$ does not belong to the interior of $B$ as for any $\delta>0$ the element $\delta/2 \cdot 1_A\in B(0,\delta)$ is invertible.
\end{proof}

\begin{proposition}\label{weaklyopen}
Let $A$ be a unital Banach algebra. Then multiplication is open at points in the sets $A\times\mathrm{GL}\, A, \, \mathrm{GL}\, A\times A.$
\end{proposition}
\begin{proof}
Take $x,y\in A$ where $x$ is invertible. Given $\varepsilon>0$, set $\delta=\varepsilon/2\|x^{-1}\|$. We \emph{claim} that $B(xy,\delta)\subseteq B(x,\varepsilon)\cdot B(y,\varepsilon)$. Consider $z\in B(xy,\delta)$; plainly, $z=x\cdot x^{-1}z$ and $x\in B(x,\varepsilon)$. Moreover,
$$
\big\|x^{-1}z-y\big\|\leqslant\big\|x^{-1}\big\|\!\cdot\!\|z-xy\|\leqslant\frac{\varepsilon}{2}<\varepsilon.
$$
The other case is analogous.\end{proof}

\begin{corollary}\label{weaklyopendense}Suppose that $A$ is a unital Banach algebra for which ${\rm GL}\, A$ is dense in $A$. Then $A$ has weakly open multiplication. \end{corollary}

The converse statement is false as multiplication in the real algebra $C[0,1]$ is weakly open \cite[Theorem 5]{bww}, however invertible elements are not dense in $C[0,1]$.

\subsection{Openness of multiplication in $C(X)$ and in semigroup algebras}
\begin{proposition}\label{dim2}Let $X$ be a compact Hausdorff space. If $\dim X \geqslant 2$, then multiplication in $C(X)$ is not open. \end{proposition}
\begin{proof}The complex algebra $C(X)$ has dense invertible group if and only if the covering dimension of $X$ is at most 1 (see, \emph{e.g.}, \cite[Proposition 1.7]{rieffel}) (and in the real case this is if and only if $\dim X <1$). We are now in a position to apply Proposition~\ref{multiplication}. \end{proof}

\begin{proposition}Let $k\in \mathbb{N}$ and denote by $\ell_\infty^k$ the $k^{{\rm th}}$ power of the scalar field endowed with the maximum norm. Then $\ell_\infty^k$ has uniformly open multiplication with $\delta(\varepsilon) = \varepsilon^2 / 4$ $(\varepsilon > 0)$. \end{proposition}
\begin{proof}Multiplication in the scalar field is uniformly open with $\delta(\varepsilon) = \varepsilon^2 / 4$. The algebra $\ell_\infty^k$ is isometrically isomorphic to $\ell_\infty(\Gamma, A)$, where $\Gamma$ is a $k$-element set and $A$ is the scalar field. The claim now follows from Lemma~\ref{sums}. \end{proof}

Komisarski \cite{komisarski} proved if $X$ is a zero-dimensional compact Hausdorff space, then multiplication in the real algebra $C(X)$ is open. Actually it follows from his proof that it is uniformly open with $\delta(\varepsilon)=\varepsilon^2/4$. However, the proof does not  carry over directly to complex scalars. Here we offer a simple proof that works both for real and complex scalars. 
\begin{proposition}\label{0dim}Let $X$ be a zero-dimensional compact Hausdorff space. Then $C(X)$ has uniformly open multiplication. \end{proposition}
\begin{proof}Let $\mathcal{B}$ be the field of all subsets of $X$ that are closed and open. For every finite subset $F\subseteq \mathcal{B}$ let $\mathcal{B}_F$ be the subfield of $\mathcal{B}$ generated by $F$. Then $A_F={\rm span}\{\mathds{1}_D\colon D\in \mathcal{B}_F\}$ is a~unital subalgebra of $C(X)$ which is isometrically isomorphic to $\ell_\infty^k$, where $k=\log_2 |\mathcal{B}_F|$. The algebra $A_0 = \bigcup_{F\subseteq \mathcal{B}, {\rm finite}} A_F$ is dense in $C(X)$ and it has uniformly open multiplication with $\delta(\varepsilon)=\varepsilon^2/4$ ($\varepsilon >0$). In order to complete the proof, we apply Proposition~\ref{completion}.\end{proof}

\begin{proposition}Let $S$ be a cancellative monoid and set $A=\ell_1(S)$. If ${\rm tsr}\, A =1$, then $S$ is a group. Consequently, if $S$ is not a group, multiplication in $A$ is not open. \end{proposition}
\begin{proof}Denote by 1 the neutral element in $S$. Assume contrapositively that $S$ is not a~group. Let $g\in S$ be an element that does not have either right or left inverse. Let us say that $g$ does not have a right inverse (the other case is symmetric). Consequently, $e_g$ is not (right) invertible in $A$ as if it were, with $e_g \ast x = e_1$, where $x=(\xi_s)_{s\in S}\in \ell_1(S)$, then
$ e_g \ast x = \sum_{s\in S}\xi_s e_{gs}=e_1,$ so $\xi_s e_{gs}=e_1$ for some $s\in S$, thus $\xi_s=1$ and $s$ is a right inverse for $g$. \smallskip

Moreover, as $S$ is cancellative the map $s\mapsto gs$ ($s\in S$) is injective and so
$$\|e_g \ast x\|=\Big\|\sum_{s\in S}\xi_s e_{gs} \Big\| = \sum_{s\in S} |\xi_s| = \|x\|,$$
whence $e_g$ is not a topological zero divisor. By Lemma~\ref{topzero}, ${\rm GL}\,A$ is not dense in $A$. We are now in a position to apply Proposition~\ref{multiplication}.\end{proof}

Having prepared necessary technology, we may answer \cite[Question 3 on p.~493]{bbs} in the negative.

\begin{corollary}\label{Cauchyproduct}The Cauchy product in $\ell_1(\mathbb{N}_0)$ is not an open map. \end{corollary}

\begin{remark}Let $S$ be a subset of the real line that contains the smallest element. Then $S$ is a commutative monoid when endowed with the operation of taking maximum. When $S$ contains at least two elements, it is not a group. Let us note that when $S$ is a doubleton, $A$ is isomorphic to $C(\{1,2\})$, which obviously has open multiplication. This example demonstrates that one cannot remove the hypothesis of cancellativity of $S$.

\end{remark}
Employing the Fourier transform, we may apply Proposition~\ref{dim2} to certain group algebras.
\begin{corollary}Let $G$ be a discrete Abelian group. Suppose that the covering dimension of the dual group $\widehat{G}$ is at least 2. Then multiplication in the complex algebra $A=\ell_1(G)$ is not open. \end{corollary}

\begin{proof}The Fourier transform $\mathcal{F}\colon \ell_1(G)\to C(\widehat{G})$ is a homomorphism of algebras. By the Riemann--Lebesgue lemma, it has dense range. If ${\rm GL}\,\ell_1(G)$ were dense in~$\ell_1(G)$, $\mathcal{F}[{\rm GL}\,\ell_1(G)]\subseteq {\rm GL}\,C(\widehat{G})$ would be dense in $C(\widehat{G})$, however this is prevented by the dimen\-sional constraint.  \end{proof}

\begin{corollary}\label{Zd}For any $d>1$, convolution in $\ell_1(\mathbb{Z}^d)$ is not open. \end{corollary}
\begin{proof}By standard harmonic analysis, $\widehat{\mathbb{Z}^d} = \mathbb{T}^d$, and the product of $d$ copies of the unit circle $\mathbb T$ has dimension $d>1$. \end{proof}

\begin{theorem}\label{cyclic}Let $\mathbb{Z}_n = \mathbb{Z}/n\mathbb{Z}$. Then multiplications (convolutions) in $\ell_1(\mathbb{Z}_n)$ $(n\in \mathbb{N})$ are not equi-uniformly open.\end{theorem}

\begin{proof}Assume that multiplications in $\ell_1(\mathbb{Z}_n)$ ($n\in \mathbb N$) are equi-uniformly open. Let us fix a non-principal ultrafilter $\mathcal{U}$ on $\mathbb{N}$ and let $$G={\prod_{n\in \mathbb{N}}}^{\mathcal U} \mathbb{Z}_{n!}.$$ Let us say that an element $[(g_n)_{n\in \mathbb N}]\in G$ is \emph{bounded} if for some $C>0$ it has representatives $(\widehat{g}_n)_{n\in \mathbb{N}}$ such that $-n!/2 \leqslant \widehat{g}_n \leqslant n!/2$ and $|\widehat{g}_n| \leqslant C$ for all $n$. Denote by $G_{\rm bdd}$ the subgroup of $G$ comprising all bounded elements. As remarked by A.~Thom \cite{thom}, the function
$$L\colon G_{{\rm bdd}}\to \mathbb{Z}\subseteq \mathbb Q, \qquad L\big([(g_n)_{n\in \mathbb N}]\big) = \lim_{n\to \mathcal{U}} \widehat{g}_n \quad \big([(g_n)_{n\in \mathbb N}]\in G_{{\rm bdd}}\big) $$
is a group homomorphism, which by injectivity of the additive group of rationals, extends to a group homomorphism $\theta \colon G\to \mathbb{Q}$ \cite[Proposition 24.5]{fuchs1}. By a result of S\k{a}siada \cite{sasiada} (see \cite[Proposition 94.2]{fuchs2} for a modern exposition), the image of $\theta$, being countable, must contain a non-trivial divisible subgroup, which implies that $\theta$ is surjective.\smallskip

By Lemma~\ref{ultraproductlemma} and the assumption that $\ell_1(\mathbb{Z}_n)$ ($n\in \mathbb N$) have equi-uniformly open multiplications, the algebra $\prod_{n\in\mathbb{N}}^{\mathcal U} \ell_1(\mathbb{Z}_{n!})$ has uniformly open multiplication. Since $h_G$ given by \eqref{homomorphismG} is surjective, $A=\ell_1(G)$ has uniformly open multiplication. As the map $h_\theta$ given by $\eqref{homomorphismtheta}$ is surjective, according to Lemma~\ref{quotient}, $\ell_1(\mathbb{Q})$ has uniformly open multiplication. Reasoning as before, we conclude that for any ultrafilter $\mathcal V$, $\ell_1(\mathbb Q^{\mathcal{V}})$ has uniformly open multiplication, where $\mathbb Q^{\mathcal{V}}$ denotes the ultrapower of $\mathbb Q$ with respect to $\mathcal V$. We observe that $\mathbb Q$ and $\mathbb R$ have isomorphic ultrapowers. \smallskip

Indeed, let us regard $\mathbb Q$ and $\mathbb R$ as vector spaces over $\mathbb Q$. Any of their ultrapowers retains the structre of a vector space over rationals. Two vector spaces over the the same field are isomorphic if and only if they have bases of the same cardinality so we need to find the corresponding dimensions of ultrapowers. We have $\dim_{\mathbb Q}\mathbb R = \mathfrak c$, the cardinality of the continuum, and as for any non-principal ultrafilter $\mathcal{V}$ on a set of cardinality $\mathfrak c$, $\mathbb R^{\mathcal{V}}$ has cardinality $\mathfrak c$, hence $\dim_{\mathbb Q}\mathbb R^{\mathcal{V}} = \mathfrak c$. For any such ultrafilter, we have also $\dim_{\mathbb Q}\mathbb  Q^{\mathcal{V}} = \mathfrak c$, which means that $\mathbb Q^{\mathcal{V}}$ and $\mathbb R^{\mathcal{V}}$ are isomorphic as vector spaces, so also as Abelian groups.\smallskip

Consequently, we may suppose that $\mathbb{R}$ embeds into $\mathbb{Q}^{\mathcal{V}}$. However, the group of reals is divisible, so its copy in $\mathbb{Q}^{\mathcal{V}}$ is a direct summand \cite[Proposition 24.5]{fuchs1}, hence $\mathbb R$ is a homomorphic image of $\mathbb{Q}^{\mathcal{V}}$. Consequently, by Lemma~\ref{quotient}, $\ell_1(\mathbb{R})$ has uniformly open multiplication. However, it is known that $\ell_1(\mathbb{R})$ does not have topological stable rank 1 (neither in the real nor in the complex case, see, \emph{e.g.}, \cite[Theorem 1.2]{misa}), so this is a~contradiction to Proposition~\ref{multiplication}.\end{proof}

\begin{corollary}\label{ZandQ}The algebras $\ell_1(\mathbb{Z})$ and $\ell_1(\mathbb{Q})$ do not have uniformly open multiplication. \end{corollary}
\begin{proof}It is enough to notice that each cyclic group is a quotient of $\mathbb{Z}$ and apply Lemma~\ref{quotient}. The assertion for $\ell_1(\mathbb{Z})$ follows from the final part of the proof of Theorem~\ref{cyclic}. \end{proof}

We conclude this section by considering the space ${\rm BV}[0,1]$ of all bounded, scalar-valued functions $f$ on $[0,1]$ that have bounded total variation $V_0^1f$. Then ${\rm BV}[0,1]$ becomes a~unital Banach algebra under the pointwise product and the norm $$\|f\|_{{\rm BV}} = \sup_{t\in [0,1]}|f(t)|+V_0^1f.$$
Even though not required for our purposes, ${\rm BV}[0,1]$ is isometric to a measure convolution algebra on a certain topological semigroup.\smallskip

Every element in ${\rm BV}[0,1]$ may be approximated by invertible elements as already proved by Kowalczyk and Turowska \cite[Lemma~3.1]{kt}. (In the case of complex scalars we apply their reasoning to the real and the imaginary part of $f$ separately.) By appealing to Corollary~\ref{weaklyopendense}, we obtain the the main result of \cite{kt}.
\begin{theorem}The algebra ${\rm BV}[0,1]$ has weakly open multiplication. \end{theorem}
Let us remark that the spectrum of ${\rm BV}[0,1]$ is zero-dimensional \cite[Theorem 5]{bluminger}, hence it is very likely that ${\rm BV}[0,1]$ has actually uniformly open multiplication.

\subsection{Banach algebras of operators on Banach spaces and their Calkin algebras}\label{calkin} 

Suppose that $E$ is a Banach space that admits a bounded linear operator $T\colon E\to E$ with the property that for some $\gamma > 0$ we have $\|Tx\|\geqslant \gamma \|x\|$ $(x\in E)$ and $T[E]$ is a proper subspace of $E$. If $T[E]$ is complemented in $E$ (one may find such operator, for instance, in the case where $X$ is isomorphic to its hyperplances; every classical Banach space has this property), then multiplication in $\mathcal{B}(E)$, the algebra  of all bounded linear operators on $E$, is not open. This partially answers \cite[Question 4 on p.~493]{bbs}. \smallskip

Indeed, as $T$ is bounded below and $T[E]$ is complemented, there is an operator $S\in \mathcal{B}(E)$ such that $ST={\rm id}_E$. However, being non-surjective, $T$ is not invertible. It follows now from Proposition~\ref{rieffel} that invertible elements in $\mathcal{B}(E)$ are not dense. \smallskip

By Atkinson's theorem (see, \emph{e.g.}, \cite[Theorem 3.2.2.]{arveson}), invertible elements in the Calkin algebra $\mathcal{Q}(E)$ are exactly images of Fredholm operators in $\mathcal{B}(E)$ via the canonical quotient map. Thus if one may find $T$ with the additional property that the quotient space $E / T[E]$ is infinite-dimensional, then multiplication in the Calkin algebra $\mathcal{Q}(E)$ is not open (this is still the case, for instance when $E$ is isomorphic to its Cartesian square $E\oplus E$; and this property holds for all classical Banach spaces too). Let us record our findings formally.

\begin{theorem}\label{B(E)andQ(E)}Suppose that $E$ is a Banach space which contains a~proper subspace $F$ that is isomorphic to $E$ and complemented. Then $\mathcal{B}(E)$ does not have open multiplication. If moreover $F$ may be chosen to have infinite codimension in $E$, then the Calkin algebra $\mathcal{Q}(E)$ does not have open multiplication either.\end{theorem}

In the case where $E$ is isomorphic to $E\oplus E$, the above theorem may be deduced from a result of Corach and Larotonda \cite{cl}, who proved that in such case the stable rank of $\mathscr{B}(E)$ is infinite. We note in passing that there exist spaces $E$ not isomorphic to their hyperplanes for which the Calkin algebra $\mathcal{Q}(E)$ is one-dimensional \cite{ah} or isomorphic to $C(X)$ for any countable compact Hausdorff space \cite{mpp}, hence it has (uniformly) open multiplication (Proposition~\ref{0dim}). On the other hand, there also exists a space $E$ that is not isomorphic to its hyperplanes whose corresponding Calkin algebra is isometrically isomorphic to $\ell_1(\mathbb{N}_0)$ \cite[Chapter~4]{tarbard}---in which case the Calkin algebra $\mathcal{Q}(E)$ fails to have open multiplication (Lemmata~\ref{quotient} and \ref{Cauchyproduct}).\smallskip

Let us elaborate more on the case where $\mathcal{Q}(E)$ has topological stable rank 1. Again, by Atkinson's theorem, the preimage $\kappa^{-1}[{\rm GL}\, \mathcal{Q}(E)]$ is the (open) set of all Fredholm operators on $E$, where $\kappa\colon \mathcal{B}(E)\to \mathcal{Q}(E)$ is the canonical quotient map. Thus, by the properties of the quotient norm, $\kappa^{-1}[{\rm GL}\, \mathcal{Q}(E)]$ is also dense in $\mathcal{B}(E)$. Let us make an extra assumption that every Fredholm operator on $E$ has Fredholm index 0 (this is the case, for example when $E$ is hereditarily indecomposable). By Lemma~\ref{index0approx}, every index-zero Fredholm operator is approximable by invertible operators, so in our case ${\rm GL}\, \mathcal{B}(E)$ is dense in $\mathcal{B}(E)$. By Proposition~\ref{weaklyopen}, $\mathcal{B}(E)$ has weakly open multiplication. We have thus proved the following result.

\begin{theorem}Suppose that $E$ is a Banach space with the property that every Fredholm operator in $\mathcal{B}(E)$ has index $0$. If ${\rm tsr}\,\mathcal{Q}(E)=1$, then both $\mathcal{B}(E)$ and $\mathcal{Q}(E)$ have weakly open multiplication.\end{theorem}

The above theorem applies for instance to the spaces constructed by Motakis, Puglisi and Zisimopoulou \cite{mpp} whose Calkin algebras are of the form $C(X)$ for a countable, compact space $X$, and is related to the question of which other $C(X)$-spaces may be realised as Calkin algebras. By Komisarski's result \cite{komisarski}, if $\dim X \geqslant 2$, multiplication in the real algebra $C(X)$ is not weakly open. This suggests that should this be possible for at least 2-dimensional compacta, completely different methods for the construction of such $E$ must be employed. Proposition~\ref{index0} gives a more handy criterion for weak openness of multiplication. 

\begin{theorem}\label{B(E)weaklyopen}Suppose that $E$ is a complex Banach space with the property that each operator in $\mathcal{B}(E)$ has totally disconnected (for example, countable) spectrum. Then both $\mathcal{B}(E)$ and $\mathcal{Q}(E)$ have weakly open multiplication.\end{theorem}

Every hereditarily indecomposable Banach space satisfies the hypothesis of the above theorem as operators acting on such spaces have countable spectra. Nevertheless, Read (\cite{read}) has constructed a Banach space $E_{\rm R}$ with the property that $\mathcal{B}(E_{\rm R})$ admits the unitisation of $\ell_2$ endowed with the trivial product as a quotient. As such, $\mathcal{B}(E_{\rm R})$ cannot have weakly open multiplication. We may apply Proposition~\ref{index0} to algebras of compact operators as they have countable spectra.

\begin{corollary}Let $E$ be a Banach space and let $\mathcal{K}^\sharp(E)$ denote the unital subalgebra of $\mathcal{B}(E)$ generated by compact operators on $E$. Then ${\rm tsr}\,\mathcal{K}^\sharp(E)=1$ and so $\mathcal{K}^\sharp(E)$ has weakly open multiplication.\end{corollary}

\section{Open problems}
\begin{enumerate}
\item Is multiplication (uniformly) open in the complex algebra $C(X)$ in the case where $X$ is an arbitrary one-dimensional compact Hausdorff space, \emph{e.g.}, the circle or the Pontryagin dual of the group of rational numbers?
\item What are other permanence properties of algebras with (weakly, uniformly) open multiplication? 
\begin{itemize}
\item Let $A$ be a Banach algebra with (weakly, uniformly) open multiplication. Does $A^{**}$ endowed with either Arens product have (weakly, uniformly) open multiplication? Note that $A^{**}$ may have uniformly open multiplication even if $A$ does not. Indeed, for any compact Hausdorff space $X$, both Arens products on $C(X)^{**}$ coincide and the latter algebra is isometrically isomorphic to $C(Y)$ for some zero-dimensional compact Hausdorff space $Y$, hence it has uniformly open multiplication (Proposition~\ref{0dim}).
\item Do extensions of Banach algebras preserve (weakly, uniformly) open multiplication? More precisely, if $J$ is a closed ideal of $A$ and both $J$ and $A/J$ have (weakly, uniformly) open multiplication, must it be also the case for $A$?
\item Does the projective tensor product of two Banach algebras with (weakly, uniformly) open multiplication have (weakly, uniformly) open multiplication? \end{itemize}
\item Do commutative, unital Banach algebras with zero-dimensional maximal ideal space have (uniformly) open multiplication? What if idempotents in the algebra are linearly dense?
\end{enumerate}

\subsection{Acknowledgements} The second-named author is grateful to Marek Balcerzak for bringing the question of openness of convolution to his attention and for warm hospitality during his visit to {\L}\'{o}d\'{z} (thanks are extended to W{\l}odzimierz Fechner, Szymon G{\l}\k{a}b and Micha{\l} Pop{\l}awski). It is our pleasure to thank Yemon Choi and Niels Laustsen (Lancaster) and Filip Strobin ({\L}\'{o}d\'{z}) for conversations concerning various aspects of this work. Finally, we wish to thank the referee for the remarks that improved presentation of the paper.


\begin{thebibliography}{99}
\bibitem{ah}  S.A. Argyros and R.G. Haydon, A hereditarily indecomposable $\mathscr{L}_\infty$-space that solves the scalar-plus-compact
problem, \emph{Acta Math.} \textbf{206} (2011), 1--54.
\bibitem{arveson} W.~Arveson, \emph{A short course on spectral theory}. Springer-Verlag, New York 2002.
\bibitem{bbs} M.~Balcerzak, E.~Behrends and F.~Strobin, On certain uniformly open multilinear mappings, \emph{Banach J. Math. Anal.} \textbf{10} (2016), no. 3, 482--494.
\bibitem{bms} M.~Balcerzak, A.~Majchrzycki and F.~Strobin, Multiplication in Banach spaces and multilinear continuous functionals, preprint (2013), \texttt{arXiv:1309.3433}.
\bibitem{bmw} M.~Balcerzak, A.~Majchrzycki and W.~Wachowicz, Openness of multiplication in some function spaces, \emph{Taiwanese J. Math.} \textbf{17}, No. 3 (2013), 1115--1126.
\bibitem{bm} M. Balcerzak and A. Maliszewski, On multiplication in spaces of continuous functions, \emph{Colloq. Math.}~\textbf{122} (2011) 247--253.
\bibitem{bww} M.~Balcerzak, W.~Wachowicz and W. Wilczy\'{n}ski, Multiplying balls in the space of continuous functions on $[0,1]$. \emph{Studia Math.} \textbf{170} (2) (2005), 203--209.
\bibitem{be1} E.~Behrends, Walk the dog, or: products of open balls in the space of
continuous functions, \emph{Func. Approx. Comment. Math.} \textbf{44} (2011), 153--164.
\bibitem{be2} E.~Behrends, Products of $n$ open subsets in the space of continuous functions on $[0, 1]$, \emph{Studia Math.} \textbf{204} (1) (2011), 73--95.
\bibitem{be3} E.~Behrends, Where is matrix multiplication locally open?, \emph{Linear Algebra Appl.} \textbf{517} (2017), 167--176.
\bibitem{be4} E.~Behrends, Where is pointwise multiplication in real $CK$-spaces locally open?, \emph{Fund.~Math.}~\textbf{236} (2017), 51--69.
\bibitem{bluminger} M.~Bl\"{u}minger, Topological algebras of functions of bounded variation, II, \emph{Manuscripta Math}.~\textbf{65} (1989), 377--384.
\bibitem{bd} F.F.~Bonsall and J.~Duncan, \emph{Complete normed algebras}, Springer-Verlag, Berlin 1973.
\bibitem{cohen} P.J.~Cohen, A counterexample to the closed graph theorem for bilinear maps, \emph{J. Funct.~Anal.}~\textbf{16} (1974), 235--239.
\bibitem{cl} G.~Corach and A.R.~Larotonda, Le rang stable de certaines algèbres d'op\'{e}rateurs, \emph{C. R. Acad. Sci. Paris S\'{e}r. I Math}. \textbf{296} (1983), no. 23, 949--951.
\bibitem{cg} G.~Corach and F.D.~Su\'{a}rez, Thin spectra and stable range conditions, \emph{J. Funct. Anal.} \textbf{81} (1988), no. 2, 432--442.
\bibitem{df} H.G.~Dales and J.F.~Feinstein, Banach function algebras with dense invertible group, \emph{Proc. Amer. Math.~Soc.}~\textbf{136}, No. 4 (2008), 1295--1304.
\bibitem{dls} H.G.~Dales, A.T.-M. Lau and D. Strauss, {Banach algebras on semigroups and on their compactifications}, \emph{Mem.~Amer.~Math.~Soc.}~\textbf{205} (2010), 1--165.
\bibitem{daws} M.~Daws, Amenability of ultrapowers of Banach algebras, \emph{Proc. Edinb. Math. Soc.} (2) \textbf{52}
(2009), 307--338.
\bibitem{fuchs1} L.~Fuchs, \emph{Infinite abelian groups}. Vol. I, Pure and Applied Mathematics, Vol. 36, New York: Academic Press (1970).
\bibitem{fuchs2} L.~Fuchs, \emph{Infinite abelian groups}. Vol. II, Pure and Applied Mathematics, Vol. 36-II, New York: Academic Press (1973).
\bibitem{horowitz} C.~Horowitz, An elementary counter-example to the open mapping principle for bilinear maps, \emph{Proc. Amer. Math. Soc.}~\textbf{53} (1975), 293--294.
\bibitem{komisarski} A.~Komisarski, A connection between multiplication in $C(X)$ and the dimension of $X$, \emph{Fund. Math.} \textbf{189} (2) (2006), 149--154.
\bibitem{kt} S.~Kowalczyk and M.~Turowska, Openness and weak openness of multiplication in the space of functions of bounded variation, \emph{J. Math. Anal. Appl.}~\textbf{448} (2017) 1560--1571.
\bibitem{mv} R.~Meise and D.~Vogt, \emph{Introduction to Functional Analysis}, Clarendon Press, Oxford 1997.
\bibitem{misa} K.~Mikkola and A.J.~Sasane, Bass and topological stable ranks of complex and real algebras of measures, functions and sequences, \emph{Complex Anal. Oper. Theory}~\textbf{4} (2) (2010), 401--448.
\bibitem{mpp} P.~Motakis, D.~Puglisi and D.~~Zisimopoulou, A hierarchy of separable commutative Calkin algebras, \emph{Indiana Univ. Math. J.}~\textbf{65} (2016), 39--67.
\bibitem{read} C.J. Read, Discontinuous derivations on the algebra of bounded operators on a Banach space, \emph{J. London
Math. Soc.} \textbf{40} (1989), 305--326.
\bibitem{rieffel} M.A.~Rieffel, Dimension and stable rank in the $K$-theory of $C^*$-algebras, \emph{Proc.~Lond.~Math.~Soc., III.~Ser}.~\textbf{46} (1983), 301--333.
\bibitem{robertson} A.G.~Robertson, Stable range in $C^\ast$-algebras, \emph{Math. Proc. Cambridge Philos. Soc.}~\textbf{87} (1980), 413--418.
\bibitem{sasiada} E.~S\k{a}siada, Proof that every countable and reduced torsion-free abelian group is slender, \emph{Bull.
Acad. Polon. Sci.}~\textbf{7} (1959), 143--144.
\bibitem{tarbard} M.~Tarbard, Operators on Banach Spaces of Bourgain--Delbaen Type, St John's College, University of Oxford, (2013), \texttt{arXiv:1309.7469v1}.
\bibitem{thom} A.~Thom, Ultraproducts of finite cyclic groups, \texttt{https://mathoverflow.net/a/37250/15129}, (2010).
\end{thebibliography}
\end{document}